\documentclass[draft,reqno,11pt,oneside]{amsart}
\usepackage{amsmath,amssymb,amsfonts}
\usepackage{tikz-cd}
\usepackage{amsbsy}
\usepackage{amscd}
\usetikzlibrary{matrix,arrows,decorations.pathmorphing}
\usepackage{graphicx}
\usepackage{float}
\usepackage{enumitem}
\usepackage{setspace}
\usepackage[margin=1.0in]{geometry}
\usepackage[all]{xy}

\newtheorem{thm}{Theorem}[section]
\newtheorem{lem}[thm]{Lemma}
\newtheorem{cor}[thm]{Corollary}
\newtheorem{prop}[thm]{Proposition}

\theoremstyle{definition}
\newtheorem{example}[thm]{Example}
\theoremstyle{definition}

\theoremstyle{definition}
\newtheorem{defn}[thm]{Definition}
\theoremstyle{definition}
\newtheorem{remark}[thm]{Remark}

\newcommand{\mc}[1]{\mathcal{#1}}
\newcommand{\e}[1]{\emph{#1}}

\newcommand{\la}{\langle}
\newcommand{\ra}{\rangle}

\newcommand{\rmv}[1]{}

\newcommand{\LO}{L^1(G)}

\newcommand{\LT}{L^2(G)}
\newcommand{\LTc}{L^2(G)_c}

\newcommand{\LI}{L^{\infty}(G)}

\newcommand{\BH}{\mc{B}(H)}
\newcommand{\Th}{\mc{T}(H)}

\newcommand{\BLT}{\mc{B}(L^2(G))}

\newcommand{\TC}{\mc{T}(L^2(G))}

\newcommand{\NCBLTD}{\mc{CB}^{\sigma,VN(G)}_{L^{\infty}(G)}(\mc{B}(L^2(G)))}

\newcommand{\Mcb}{M_{cb}A(G)}
\newcommand{\Qcb}{Q_{cb}(G)}

\newcommand{\vphi}{\varphi}

\newcommand{\lm}{\lambda}

\newcommand{\Gam}{\Gamma}
\newcommand{\om}{\omega}

\newcommand{\ten}{\otimes}
\newcommand{\oten}{\overline{\otimes}}
\newcommand{\pten}{\widehat{\otimes}}
\newcommand{\hten}{\otimes^h}
\newcommand{\iten}{\otimes^{\vee}}

\newcommand{\whten}{\otimes^{w^*h}}

\newcommand{\id}{\textnormal{id}}
\newcommand{\ev}{\textnormal{ev}}
\newcommand{\h}[1]{\widehat{#1}}
\newcommand{\wh}[1]{\widehat{#1}}

\newcommand{\Ad}{\mathrm{Ad}}

\newcommand{\ep}{\varepsilon}

\providecommand{\norm}[1]{\lVert#1\rVert}

\newcommand{\C}{\mathbb{C}}
\newcommand{\N}{\mathbb{N}}
\newcommand{\Z}{\mathbb{Z}}

\newcommand{\T}{\mathbb{T}}

\DeclareSymbolFont{lettersA}{U}{txmia}{m}{it}
\DeclareMathSymbol{\W}{\mathord}{lettersA}{151}

\begin{document}

\title[]{A non-commutative Fej\'{e}r theorem for crossed products, the approximation property, and applications} 
\author{Jason Crann$^1$}
\email{jasoncrann@cunet.carleton.ca}
\author{Matthias Neufang$^{1,2}$}
\email{matthias.neufang@carleton.ca}
\address{$^1$School of Mathematics and Statistics, Carleton University, Ottawa, ON, Canada H1S 5B6}
\address{$^2$Universit\'{e} Lille 1 - Sciences et Technologies, UFR de Math\'{e}matiques, Laboratoire de Math\'{e}\-matiques Paul Painlev\'{e} - UMR CNRS 8524, 59655 Villeneuve d'Ascq C\'{e}dex, France}

\keywords{Crossed products; locally compact groups; the approximation property; exactness}
\subjclass[2010]{47L65, 46L55, 46L07}


\begin{spacing}{1.0}

\maketitle
\begin{abstract} 
We prove that a locally compact group has the approximation property (AP), introduced by Haagerup--Kraus \cite{HK}, if and only if a non-commutative Fej\'{e}r theorem holds for its associated $C^*$- or von Neumann crossed products. As applications, we answer three open problems in the literature. Specifically, we show that any locally compact group with the AP is exact. This generalizes a result by Haagerup--Kraus \cite{HK}, and answers a problem raised by Li in \cite{Li}. We also answer a question of B\'{e}dos--Conti \cite{BC1} on the Fej\'{e}r property of discrete $C^*$-dynamical systems, as well as a question by Anoussis--Katavolos--Todorov \cite{AKT2} for all locally compact groups with the AP. In our approach, we develop a notion of Fubini crossed product for locally compact groups and a dynamical version of the slice map property.
\end{abstract}

\section{Introduction} 
In his seminal work on Fourier series in the early 20$^{th}$ century, Fej\'{e}r established, under appropriate conditions, the approximation of a function by the Ces\`{a}ro sum of its Fourier series \cite{F}. More specifically, if $f\in L^\infty(\T)$ with Fourier series $S_N(f)(t)=\sum_{n=-N}^N \hat{f}(n)e^{int}$, then
\begin{equation}\label{e:Ces}\frac{1}{N}\sum_{n=0}^{N-1}S_n(f)=F_N\ast f\rightarrow f\end{equation}
weak* (uniformly if $f \in C(\T)$), where $F_N(t)=\frac{1}{N}\sum_{n=0}^{N-1}\sum_{k=-n}^ne^{ikt}$ is Fej\'{e}r's kernel. Shortly after he gave explicit examples of continuous periodic functions whose Fourier series do not converge pointwise \cite{F2}.  

We may interpret the Ces\`aro convergence (\ref{e:Ces}) through pointwise multiplication under the Fourier transform: the sequence $(\hat{F}_N)$ forms a bounded approximate identity for the Fourier algebra $A(\Z)$, and we have $\hat{F}_N\cdot x\rightarrow x$
weak* for any $x\in VN(\Z)\cong L^\infty(\T)$ (uniformly if $x \in C^*_\lambda(\Z)\cong C(\T)$), where $\cdot$ is the canonical pointwise action of $A(\Z)$ on $VN(\Z)$. It follows that 
$$x=w^*-\lim_N\frac{1}{N}\sum_{n=0}^{N-1}\chi_{[-n,n]}\cdot x=w^*-\lim_N\frac{1}{N}\sum_{n=0}^{N-1}\sum_{k=-n}^n\tau(x\lm(k)^*)\lm(k)$$
provides a Fej\'er representation for any $x$ in the von Neumann crossed product $VN(\Z)=\Z\bar{\rtimes}\C$ through an explicit linear combination of translation operators, where $\tau$ is the canonical tracial state on $VN(\Z)$. Note the appearance of a F{\o}lner sequence for $\Z$, linking the Ces\`aro summability to amenability of $\Z$.

Similar Fej\'{e}r type representations exist for non-trivial crossed products $G\bar{\rtimes} M$ where $G$ is a locally compact abelian group acting on a von Neumann algebra $M$ (see, for instance, \cite{L,ZM} or \cite[\S7.10]{Ped}), where the coefficients are now $M$-valued. In this setting, the Ces\`aro sum is replaced by a suitable average over a bounded approximate identity in $L^1(\widehat{G})\cong A(G)$, again linking Fej\'{e}r representations to amenability of $G$.

If $G$ is a discrete group acting on a von Neumann algebra $M$, then the ``Fourier series''
\begin{equation}\label{e:FS} \sum_{s\in G} E(xu(s)^*)u(s)\end{equation}
is in general not strongly or even weak* convergent to $x$ for every $x\in G\bar{\rtimes} M$, even in the case of $\Z$ acting trivially on $\C$, as mentioned above. Here, $E:G\bar{\rtimes} M\rightarrow M$ is the canonical conditional expectation, and $u(s)$ is the image of the regular representation in the crossed product. Summability properties of (\ref{e:FS}) and related questions concerning the Fourier analysis of $C^*$- and von Neumann crossed products have been studied in detail by many authors (see, for instance, \cite{BC1,BC2,DR,Exel,G,Mc,MTT,MSTT,Mercer,Suz,Zach1,ZM}). In particular, Fej\'{e}r type representations for elements of crossed products have been considered over (weakly) amenable discrete groups \cite{BC1,BC2,Exel,G,ZM}. In this paper, we complete one aspect of this line of work by showing that Fej\'{e}r representability in crossed products is equivalent to the approximation property (AP) of the underlying locally compact group, as introduced by Haagerup--Kraus \cite{HK}. More precisely, we establish an explicit non-commutative Fej\'{e}r representation for elements of $C^*$- and von Neumann crossed products over arbitrary locally compact groups $G$ with the AP. Conversely, if every element of every ($C^*$- or von Neumann) crossed product admits such a Fej\'{e}r representation, then $G$ necessarily has the AP. (See Theorems \ref{t:Landstad} and \ref{t:cstarLandstad}.) 

As applications of our Fej\'{e}r representation, we
\begin{enumerate}
\item generalize a result by Haagerup--Kraus \cite{HK} for discrete groups to all locally compact groups; indeed we prove that every locally compact group with the AP is exact, thus answering Problem 9.4 (1) raised in K. Li's PhD thesis \cite{Li};
\item answer a question of B\'{e}dos--Conti \cite{BC1} on the Fej\'{e}r property of $C^*$-dynamical systems over discrete groups with the AP, thus generalizing the corresponding result of \cite{BC1} for weakly amenable groups.
\item answer a question raised by Anoussis--Katavolos--Todorov \cite{AKT2}, for all locally compact groups with the AP, on the structure of $VN(G)$-bimodules in $\BLT$, generalizing the corresponding result of \cite{AKT2} for abelian, compact, and weakly amenable discrete groups.
\end{enumerate}
The main contribution of our article, i.e., the explicit representation of elements in crossed products, will, in our view, lead to a number of additional applications, examples of which are given near the end of the paper.

Structurally, we begin in section 2 with preliminaries on operator space tensor products and dynamical systems. In section 3 we develop a notion of Fubini crossed product, and we introduce an associated dynamical notion of the slice map property for actions of locally compact groups on $C^*$- and von Neumann algebras. Section 4 is devoted to the proof of our non-commutative Fej\'{e}r respresentation, and section 5 contains the aforementioned applications.



\section{Preliminaries}

\subsection{Operator space tensor products} Throughout the paper we let $\pten$, $\ten^{\vee}$ and $\hten$ denote the operator space projective, injective, and Haagerup tensor products, respectively. Recall that for $C^*$-algebras $A$ and $B$, the injective tensor product $A\ten^{\vee}B$ coincides with the spatial/minimal tensor product $A\ten_{\min}B$ (see, e.g., \cite[Proposition 8.1.6]{ER}). The algebraic and Hilbert space tensor products will be denoted by $\ten$, the relevant product being clear from context. The weak* spatial tensor product will be denoted by $\oten$. On a Hilbert space $H$, we let $\mc{K}(H)$, $\mc{T}(H)$ and $\BH$ denote the spaces of compact, trace class, and bounded operators, respectively.

For dual operator spaces $X^*\subseteq\BH$ and $Y^*\subseteq\mc{B}(K)$, the \textit{weak*-Haagerup tensor product} $X^*\whten Y^*$ is the space of $u\in\BH\oten\mc{B}(K)$ for which there exist an index set $I$ and $(x_i)_{i\in I}\subseteq X$ and $(y_i)_{i\in I}\subseteq Y$ satisfying 
$\norm{\sum_{i}x_ix_i^*},\norm{\sum_iy_i^*y_i}<\infty$ and $u=\sum_i x_i\ten y_i$, where each sum is understood in the respective weak* topologies. Then
$$\norm{u}_{w^*h}:=\inf\{\norm{\sum_{i}x_ix_i^*},\norm{\sum_iy_i^*y_i}\mid u=\sum_i x_i\ten y_i\}$$
and the infimum is actually attained \cite[Theorem 3.1]{BS}. There are corresponding matricial norms on $M_n(X^*\whten Y^*)$ giving an operator space structure on $X^*\whten Y^*$ which is independent of the weak* homeomorphic inclusions $X^*\subseteq\BH$ and $Y^*\subseteq\mc{B}(K)$. Moreover, $(X\hten Y)^*\cong X^*\whten Y^*$ completely isometrically \cite[Theorem 3.2]{BS}. In particular, by \cite[Proposition 9.3.4]{ER}, for any Hilbert space $H$
$$\BH\cong(\mc{T}(H))^*\cong((H_c)^*\hten H_c)^*\cong H_c\whten (H_c)^*$$
completely isometrically, where $H_c$ refers to the column operator space structure on $H$. Note that the $w^*$-Haagerup tensor product of dual operator spaces coincides with the extended Haagerup tensor product (see \cite{ER1}).

Given dual operator spaces $X^*\subseteq\BH$ and $Y^*\subseteq\mc{B}(K)$, the \textit{normal Fubini tensor product} $X^*\oten_{\mc{F}}Y^*$ is given by
$$X^*\oten_{\mc{F}}Y^*=\{T\in \mc{B}(H\ten K)\mid (\om\ten\id)(T)\in Y^*, \ (\id\ten\rho)\in X^*, \ \om\in\Th, \  \rho\in\mc{T}(K)\}.$$
Clearly, $X^*\oten Y^*\subseteq X^*\oten_{\mc{F}}Y^*$. A dual operator space $X^*$ is said to have the dual slice map property if $X^*\oten Y^*=X^*\oten_{\mc{F}}Y^*$ for all operator spaces $Y$. It is known that $X^*$ has the dual slice map property if and only if $X$ has the operator space approximation property (see \cite[Theorem 11.2.5]{ER}). 

\subsection{Dynamical systems} Let $G$ be a locally compact group. The adjoint of convolution $\ast:\LO\pten\LO\rightarrow\LO$ is a co-associative co-multiplication $\Gamma:\LI\rightarrow\LI\oten\LI$ satisfying $\Gamma(f)(s,t)=f(st)$, for all $f\in\LI$. There are left and right fundamental unitaries $W,V\in\mc{B}(L^2(G\times G))$ which implement $\Gamma$ in the sense that
$$\Gamma(f)=W^*(1\ten M_f)W=V(M_f\ten 1)V^*, \ \ \ f\in\LI.$$
They are given respectively by
$$W\xi(s,t)=\xi(s,s^{-1}t), \ \ \  V\xi(s,t)=\xi(st,t)\Delta(t)^{1/2}, \ \ \ s,t\in G, \ \xi\in L^2(G\times G),$$
where $\Delta$ is the modular function of $G$. These fundamental unitaries are intimately related to the left and right regular representations $\lm,\rho:G\rightarrow\BLT$ given by
$$\lm(s)\xi(t)=\xi(s^{-1}t), \ \ \rho(s)\xi(t)=\xi(ts)\Delta(s)^{1/2}, \ \ \ s,t\in G, \ \xi\in\LT.$$
The von Neumann algebra generated by $\lm(G)$ is called the \textit{group von Neumann algebra} of $G$ and is denoted by $VN(G)$. It follows that $W\in\LI\oten VN(G)$ and $V\in VN(G)'\oten\LI$.

The set of coefficient functions of the left regular representation,
\begin{equation*}A(G)=\{\psi:G\rightarrow\mathbb{C} : \psi (s)=\la\lm(s)\xi,\eta\ra, \ \xi,\eta\in\LT, \ s\in G\},\end{equation*}
is called the \textit{Fourier algebra} of $G$. It was shown by Eymard that, endowed with the norm
$$\norm{\psi}_{A(G)}=\text{inf}\{\norm{\xi}_{\LT}\norm{\eta}_{\LT} : \psi(\cdot)=\la\lm(\cdot)\xi,\eta\ra\},$$
$A(G)$ is a Banach algebra under pointwise multiplication \cite[Proposition 3.4]{E}. Furthermore, it is the predual of $VN(G)$, where the duality is given by
\begin{equation*}\la \psi,\lm(s)\ra=u(s),\ \ \ \psi\in A(G), \ s\in G.\end{equation*}
The inclusion $A(G)\subseteq VN(G)^*$ induces a canonical operator space structure on $A(G)$.


The adjoint of pointwise multiplication $\cdot:A(G)\pten A(G)\rightarrow A(G)$ defines a co-associative co-multiplication $\wh{\Gamma}:VN(G)\rightarrow VN(G)\oten VN(G)$ satisfying $\wh{\Gamma}(\lm(s))=\lm(s)\ten\lm(s)$, $s\in G$. There are left and right fundamental unitaries $\h{W},\h{V}\in\mc{B}(L^2(G\times G))$ which implement the co-multiplication via
$$\h{\Gamma}(x)=\h{W}^*(1\ten x)\h{W}=\h{V}(x\ten 1)\h{V}, \ \ \ x\in VN(G).$$
They are given specifically by 
$$\h{W}\xi(s,t)=\xi(ts,t), \ \ \ \h{V}\xi(s,t)=W\xi(s,t)=\xi(s,s^{-1}t), \ \ \ s,t\in G, \ \xi\in L^2(G\times G).$$

A function $\vphi\in\LI$ is a \textit{completely bounded multiplier} of $A(G)$ if the map
$$m_{\vphi}:A(G)\ni v\mapsto \vphi\cdot v\in A(G)$$
is completely bounded. Under the natural norm, the space $\Mcb$ of completely bounded multipliers becomes a completely contractive Banach algebra. It is known that $\Mcb$ is a dual operator space with predual $\Qcb$ \cite[Proposition 1.10]{deCH} (see also \cite{KR}). A locally compact group $G$ has the \textit{approximation property (AP)} if there exists a net $(v_i)$ in $A(G)$ such that $v_i\rightarrow 1$ $\sigma(\Mcb,\Qcb)$. By \cite[Remark 1.2]{HK}, the net $(v_i)$ may always be chosen in $A(G)\cap C_c(G)$.

Given $\vphi\in\Mcb$, the adjoint of $m_{\vphi}$ defines a normal completely bounded map on $VN(G)$. This map extends in a canonical fashion to a normal completely bounded $\LI$-bimodule map $\h{\Theta}(\vphi)$ on $\BLT$. For example, when $\vphi\in A(G)$ we have
\begin{equation}\label{e:Thetahat}\h{\Theta}(\vphi)(T)=(\vphi\ten\id)\h{W}^*(1\ten T)\h{W}, \ \ \ T\in\BLT.\end{equation}
In \cite{NRS} it was shown that the assignment $\vphi\mapsto\h{\Theta}(\vphi)$ is a completely isometric algebra isomorphism
$$\h{\Theta}:\Mcb\cong\NCBLTD$$
from $\Mcb$ onto the algebra $\NCBLTD$ of normal completely bounded $\LI$-bimodule maps on $\BLT$ which leave $VN(G)$ globally invariant.

A $W^*$-dynamical system $(M,G,\alpha)$ consists of a von Neumann algebra $M\subseteq\BH$ endowed with a homomorphism $\alpha:G\rightarrow\mathrm{Aut}(M)$ such that for each $x\in M$, the map $G\ni s\rightarrow \alpha_s(x)\in M$ is weak* continuous. Every action induces a co-action of $\LI$ on $M$, that is, a normal unital injective $*$-homomorphism $\alpha:M\rightarrow \LI\oten M$ and a corresponding right $\LO$-module structure on $M$ \cite[\S18.6]{Str}. At this level, the co-action is co-associative in the sense that $(\Gamma\ten\id)\circ \alpha = (\id\ten\alpha)\circ \alpha$,
and the module structure is determined by
$$x\star f=(f\ten\id)\alpha(x), \ \ \ f\in\LO, \ x\in M.$$
Note that the predual $M_*$ becomes a left operator $\LO$-module via $\alpha_*:\LO\pten M_*\rightarrow M_*$.

The crossed product of $M$ by $G$, denoted $G\bar{\rtimes}_{\alpha} M$, is the von Neumann subalgebra of $\BLT\oten M$ generated by $\alpha(M)$ and $VN(G)\ten 1$. When $\alpha$ is clear from context, we often simply write $G\bar{\rtimes}M$. For $\mu\in M(G)$, we write $u(\mu)=\lm(\mu)\ten 1$ for the canonical image of $\lm(\mu)$ in the crossed product. Recall that $\lm$ is involutive, so that, in particular, $u(f)^*=u(f^o)$, where
$$f^o(s)=\Delta(s)^{-1}\overline{f(s^{-1})}, \ \ \ f\in\LO.$$

Any action $\alpha$ admits a dual co-action $\h{\alpha}:G\bar{\rtimes}M\rightarrow VN(G)\oten(G\bar{\rtimes}M)$
of $VN(G)$ on the crossed product, given by
\begin{equation}\label{e:coaction}\h{\alpha}(T)=(\h{W}^*\ten 1)(1\ten T)(\h{W}\ten 1), \ \ \ T\in G\bar{\rtimes}M.\end{equation}
On the generators we have $\h{\alpha}(\hat{x}\ten 1)=\h{\Gam}(\hat{x})\ten 1$, $\hat{x}\in VN(G)$ and $\h{\alpha}(\alpha(x))=1\ten\alpha(x)$, $x\in M$. Moreover, 
$$(G\bar{\rtimes}M)^{\h{\alpha}}=\{T\in G\bar{\rtimes}M \mid \h{\alpha}(T)=1\ten T\}=\alpha(M).$$
This co-action yields a canonical right operator $A(G)$-module structure on the crossed product $G\bar{\rtimes}M$ via 
$$T\cdot\psi=(\psi\ten\id)\h{\alpha}(T), \ \ \ \psi\in A(G), \ T\in G\bar{\rtimes}M.$$
It follows from equations (\ref{e:Thetahat}) and (\ref{e:coaction}) that
$$T\cdot\psi=(\h{\Theta}(\psi)\ten\id)(T), \ \ \ \psi\in A(G), \ T\in G\bar{\rtimes}M.$$


Let $N=G\bar{\rtimes}M$, and let $\vphi:VN(G)^+\rightarrow[0,\infty]$ denote the Plancherel weight on $VN(G)$.
Then $$E=(\vphi\ten\id)\circ\h{\alpha}:N^+\rightarrow \widehat{N}^+$$
is an operator-valued weight in the sense of \cite{Haa2}, where $\widehat{N}^+$ is the extended positive part of $N$ (i.e., the set of homogeneous, additive, lower semi-continuous functions on $N_*^+$). Following \cite{L}, let 
$$\mc{M}_{(\vphi\ten\id)}^+:=\{T\in(VN(G)\oten N)^+\mid \exists \ C_T>0 \  \vphi((\id\ten\om)(T))\leq C_T\norm{\om}, \ \om\in(N_*)_+\}.$$
Also, let $\mc{N}_{(\vphi\ten\id)}:=\{T\in VN(G)\oten N\mid T^*T\in\mc{M}_{(\vphi\ten\id)}^+\}$, 
and $\mc{M}_{(\vphi\ten\id)}:=\mathrm{span} \ \mc{M}_{(\vphi\ten\id)}^+=\mc{N}_{(\vphi\ten\id)}^*\mc{N}_{(\vphi\ten\id)}$. The operator-valued weight $E$ defines a linear map from the weak*-dense subspace $N_1:=\h{\alpha}^{-1}(\mc{M}_{(\vphi\ten\id)})$ of $N$ to $N$ satisfying (see \cite[\S2, Corollary 1.3]{L} or \cite{Haa,Haa2} for details):

\begin{enumerate}
\item $E(T^*)=E(T)^*$, $T\in N_1$,
\item $E(T^*T)\geq 0$, $T\in N_1$,
\item $E(xTy)=xE(T)y$, $T\in N_1$, $x,y\in M$,
\item For $A,B\in N_0:=\h{\alpha}^{-1}(\mc{N}_{(\vphi\ten\id)})$, the map $T\mapsto E(A^*TB)$ is $\sigma$-continuous on bounded sets.
\end{enumerate}

By the proof of \cite[Lemma 2.6]{L}, $\{u(f)Tu(g)\mid T\in N, \ f,g\in C_c(G)\}\subseteq N_1$ and is weak* dense in $N$. Note that
\begin{align*}\h{\alpha}(E(u(f)Tu(g)))&=\h{\alpha}((\vphi\ten\id)\h{\alpha}(u(f)Tu(g)))\\
&=(\vphi\ten\id\ten\id)((\id\ten\h{\alpha})(\h{\alpha}(u(f)Tu(g))))\\
&=(\vphi\ten\id\ten\id)((\h{\Gamma}\ten\id)(\h{\alpha}(u(f)Tu(g))))\\
&=1\ten(\vphi\ten\id)(\h{\alpha}(u(f)Tu(g))),
\end{align*}
where the last line follows from strong right invariance of the Plancherel weight (see e.g., \cite[Proposition 3.1]{KV2}). It follows that $E$ maps $N_1$ into $\alpha(M)\cong M$, and we may view $E$ as an operator-valued weight from $N$ to $M$. When $G$ is discrete, the Plancherel weight $\vphi=\la(\cdot)\delta_e,\delta_e\ra$ and the operator-valued weight $E$ becomes the canonical faithful normal conditional expectation $E:M\bar{\rtimes} G\rightarrow M$. 

A $C^*$-dynamical system $(A,G,\alpha)$ consists of a $C^*$-algebra $A$ endowed with a homomorphism $\alpha:G\rightarrow\mathrm{Aut}(A)$ such that for each $a\in A$, the map $G\ni s\mapsto\alpha_s(a)\in A$ is norm continuous. A covariant representation $(\pi, \sigma)$ of $(A,G,\alpha)$ consists of a representation $\pi:A\rightarrow\BH$ and a unitary representation $\sigma:G\rightarrow\BH$ such that $\pi(\alpha_s(a))=\sigma(s)\pi(a)\sigma(s)^{-1}$ for all $s\in G$. Given a covariant representation $(\pi,\sigma)$, we let
$$(\pi \times \sigma)(f) = \int_G \pi(f(t)) \sigma(t) \ dt, \ \ \ f\in C_c(G,A).$$
The \emph{full crossed product} $G\rtimes_f A$ is the completion of $C_c(G,A)$ in the norm
$$\|f \| = \sup_{(\pi, \sigma)} \| (\pi \times \sigma)(f)\|$$
where the supremum is taken over  all covariant representations $(\pi, \sigma)$ of $(A,G,\alpha)$.

Let $\pi_u:A\hookrightarrow\mc{B}(H_u)$ be the universal representation of $A$. Then $(\tilde{\pi}_u, u)$ is a covariant representation on $L^2(G,H_u)$, where 
$$\tilde{\pi}_u(a)\xi(t)=\pi_u(\alpha_{t^{-1}}(a))\xi(t), \ \ \  u(s)\xi(t)=\xi(s^{-1}t), \ \ \ \xi\in L^2(G,H_u).$$
The \textit{reduced crossed product} $G\rtimes A$ is defined to be the norm closure of $\tilde{\pi}_u\times u(C_c(G,A))$.

Analogously to the von Neumann setting, one can view the action through the injective $*$-homomorphism  
$$\alpha:A\ni a\mapsto (s\mapsto\alpha_{s^{-1}}(a))\in C_b(G,A)\subseteq M(C_0(G)\iten A),$$
which, under the canonical inclusion $M(C_0(G)\iten A)\subseteq \LI\oten A^{**}$ is nothing but the representation $\tilde{\pi}_u$ above. 
Moreover, $\la\alpha(A)(C_0(G)\ten 1)\ra=C_0(G)\iten A$ 
and it follows that
$$G\rtimes A=\la\alpha(A)(C_{\lm}^*(G)\ten 1)\ra\subseteq M(\mc{K}(\LT)\iten A).$$
Letting $B=G\rtimes A$ we define, as above, 
$$\mc{A}_{(\vphi\ten\id)}^+:=\{T\in M(C^*_{\lm}(G)\iten B)^+\mid \exists \ C_T>0 \  \vphi((\id\ten\om)(T))\leq C_T\norm{\om}, \ \om\in(B^*)_+\}.$$
Also, let 
$$\mc{B}_{(\vphi\ten\id)}:=\{T\in M(C^*_{\lm}(G)\iten B)\mid T^*T\in\mc{M}_{(\vphi\ten\id)}^+\},$$
and $\mc{A}_{(\vphi\ten\id)}:=\mathrm{span} \ \mc{A}_{(\vphi\ten\id)}^+=\mc{B}_{(\vphi\ten\id)}^*\mc{B}_{(\vphi\ten\id)}$. The dual co-action in this setting is
$$\h{\alpha}:B\ni T \mapsto (\h{W}^*\ten 1)(1\ten T)(\h{W}\ten 1)\in M(C^*_\lm(G)\iten B),$$
which extends by the above formula to a normal *-homomorphism $\h{\alpha}:B''\rightarrow VN(G)\oten B''$. The operator-valued weight $E=(\vphi\ten\id)\circ\widehat{\alpha}:B''\rightarrow \widehat{B''}^+$ restricts to a linear map from $B_1:=\h{\alpha}^{-1}(\mc{A}_{(\vphi\ten\id)})$ to $B$ satisfying $E(u(f)Tu(g))\in\alpha(A)$ for every $T\in B$ and $f,g\in C_c(G)$ (see \cite[\S3]{L}).

\section{Fubini crossed products}

\subsection{$W^*$-setting}

Let $(M,G,\alpha)$ be a $W^*$-dynamical system. Viewing the crossed product $G\bar{\rtimes}M$ as a `twisted' tensor product, it is natural to study a `twisted' version of the normal Fubini tensor product.

\begin{defn}\label{d:Fubini} Let $(M,G,\alpha)$ be a $W^*$-dynamical system. For a $G$-invariant weak* closed subspace $X\subseteq M$ we define the \textit{Fubini crossed product} of $X$ by $G$ as
$$G\bar{\rtimes}_{\mc{F}}X:=\{T\in G\bar{\rtimes}M\mid E(u(f)Tu(g))\in \alpha(X), \ f,g\in C_c(G)\}.$$
\end{defn}

Based on ideas in \cite{KW}, a notion of Fubini crossed product for countable discrete groups acting on operator spaces was introduced in \cite[Definition 4.1]{UZ}. The natural analogue of \cite[Definition 4.1]{UZ} for locally compact groups acting on dual operator spaces is equivalent to our Definition \ref{d:Fubini}, as we now show.

\begin{prop}\label{p:UV} Let $(M,G,\alpha)$ be a $W^*$-dynamical system. For any $G$-invariant weak* closed subspace $X\subseteq M$,
$$G\bar{\rtimes}_{\mc{F}}X=\{T\in G\bar{\rtimes}M\mid (\rho\ten\id)(T)\in X, \ \rho\in\TC\}.$$
\end{prop}

\begin{proof} First consider $T\in G\bar{\rtimes}M$ of the form $\alpha(x)u(r)$, $x\in M$, $r\in G$. Then for $f,g\in C_c(G)$ we have, on the one hand,
\begin{align*} &E(u(f)^*Tu(g))=(\vphi\ten\id)\h{\alpha}(u(f^o)\alpha(x)u(r)u(g))\\
&=(\vphi\ten\id\ten\id)\bigg(\iint f^o(s)g(t)(\lm(s)\ten\lm(s)\ten 1)(1\ten \alpha(x))(\lm(rt)\ten\lm(rt)\ten 1) \ ds \ dt\bigg)\\
&=(\vphi\ten\id\ten\id)\bigg(\iint f^o(s)g(t)(1\ten u(s)\alpha(x)u(rt))(\lm(srt)\ten1\ten 1) \ ds \ dt\bigg)\\
&=\int f^o(s)g(r^{-1}s^{-1}) u(s)\alpha(x)u(s^{-1}) \ ds\\
&=\int \overline{f(s^{-1})}g(r^{-1}s^{-1})\Delta(s^{-1}) (1\ten\lm(s))\alpha(x)(1\ten \lm(s^{-1}) \ ds.\\
\end{align*}
Performing the substitution $s\mapsto s^{-1}$ yields
\begin{align*} E(u(f)^*Tu(g))&=\int \overline{f(s)}g(r^{-1}s)(\lm(s^{-1})\ten1)\alpha(x)(\lm(s)\ten 1) \ ds\\
&=(\overline{f}\cdot\lm(r)g\ten\id\ten\id)(\Gam\ten\id)(\alpha(x))\\
&=(\overline{f}\cdot\lm(r)g\ten\id\ten\id)(\id\ten\alpha)(\alpha(x))\\
&=\alpha(x\star(\overline{f}\cdot\lm(r)g)).
\end{align*}
On the other hand, $\lm(r)\cdot\om_{g,f}|_{\LI}=\overline{f}\cdot\lm(r)g\in\LO$, so that
$$(\om_{g,f}\ten\id)(\alpha(x)u(r))=(\lm(r)\cdot\om_{g,f}|_{\LI}\ten\id)(\alpha(x))=x\star(\overline{f}\cdot\lm(r)g).$$
Thus, 
$$\alpha((\om_{g,f}\ten\id)(T))=E(u(f)^*Tu(g)).$$
By normality the above equality is valid for every $T\in G\bar{\rtimes}M$. Since $\mathrm{span}\{\om_{g,f}\mid f,g\in C_c(G)\}$ is dense in $\TC$, the claim is established.
\end{proof}

The following example further justifies the terminology of Fubini crossed products.

\begin{example}\label{ex:1} Let $G$ be a locally compact group acting trivially on a von Neumann algebra $M$. Then $G\bar{\rtimes}M=VN(G)\oten M$. By Proposition \ref{p:UV}, for any $G$-invariant weak* closed subspace $X\subseteq M$ we have
$$G\bar{\rtimes}_{\mc{F}}X=\{T\in VN(G)\oten M\mid (\rho\ten\id)(T)\in X, \ \rho\in\TC\}=VN(G)\oten_{\mc{F}} X.$$
\end{example}

\begin{prop} Let $(M,G,\alpha)$ be a $W^*$-dynamical system. For any $G$-invariant weak* closed subspace $X\subseteq M$ we have
$$G\bar{\rtimes}X\subseteq G\bar{\rtimes}_{\mc{F}}X,$$
where $G\bar{\rtimes}X=\overline{\la\alpha(X)(VN(G)\ten 1)\ra}^{w^*}$.
\end{prop}

\begin{proof} First consider $T\in G\bar{\rtimes}X$ of the form $\alpha(x)u(r)$, $x\in X$, $r\in G$. Then, as shown in the proof of Proposition \ref{p:UV}, for $f,g\in C_c(G)$ we have
$$E(u(f)^*Tu(g))=\alpha(x\star(\overline{f}\cdot\lm(r)g)),$$
which belongs to $\alpha(X)$ by $\LO$-invariance of $X$. The normality of $E(u(f)^*(\cdot)u(g))$ ensures the same is true for arbitrary $T\in G\bar{\rtimes}X$. 
\end{proof}

Although Proposition \ref{p:UV} shows that Fubini crossed products are determined by the restriction of slice maps to the crossed product, our (equivalent) Definition \ref{d:Fubini} was motivated by considering the twisted slice map conditions
$$E(u(f)Tu(g))=(\vphi\ten\id\ten\id)(\h{\Gam}\ten\id)(u(f)Tu(g))\in\alpha(X), \ \ \ f,g\in C_c(G),$$ 
with the perspective that the Fubini crossed product is a twisted version of the Fubini tensor product. For discrete actions, the twisted slice map conditions are directly related to the $A(G)$-action on the crossed product. 

\begin{example}\label{ex:2} For any discrete $W^*$-dynamical system $(M,G,\alpha)$ and any $G$-invariant weak* closed subspace $X\subseteq M$, we have
\begin{equation}\label{e:disA(G)}G\bar{\rtimes}_{\mc{F}}X=\{T\in G\bar{\rtimes}M\mid T\cdot\psi\in G\bar{\rtimes}X, \ \psi\in A(G)\}.\end{equation}
First, observe that
\begin{align*}\la\delta_s,\lm(f)\ra&=f(s)=(\delta_{s^{-1}}\ast f)(e)=\la\vphi,\lm(\delta_{s^{-1}}\ast f)\ra\\
&=\la\vphi,\lm(s^{-1})\lm(f)\ra=\la\vphi,\lm(f)\lm(s^{-1})\ra\\
&=\la \lm(s^{-1})\cdot\vphi,\lm(f)\ra
\end{align*}
for all $f\in\LO$. Thus, $\delta_s=\lm(s^{-1})\cdot\vphi$ as elements of $A(G)$, and we have
\begin{align*}T\cdot\delta_s &=(\lm(s^{-1})\cdot\vphi\ten\id\ten\id)\h{\alpha}(T)\\
&=(\vphi\ten\id\ten\id)(\h{\alpha}(T)(\lm(s^{-1})\ten1\ten1))\\
&=(\vphi\ten\id\ten\id)(\h{\alpha}(T)(\lm(s^{-1})\ten\lm(s^{-1})\ten1)(1\ten\lm(s)\ten 1))\\
&=(\vphi\ten\id\ten\id)(\h{\alpha}(Tu(s^{-1}))(1\ten\lm(s)\ten 1))\\
&=(\vphi\ten\id\ten\id)(\h{\alpha}(Tu(s^{-1})))u(s)\\
&=E(Tu(s^{-1}))u(s), \ \ \ s\in G.\\
\end{align*}
Suppose that $T\cdot\psi\in G\bar{\rtimes}X$ for all $\psi\in A(G)$, so that $E(Tu(s^{-1}))u(s)\in G\bar{\rtimes}X$ for all $s\in G$. Since $E(G\bar{\rtimes}X)\subseteq\alpha(X)$ we see that
\begin{align*}E(Tu(s^{-1}))&=E(E(Tu(s^{-1})))=E(E(Tu(s^{-1}))u(s)u(s^{-1}))\\
&=E((T\cdot\delta_s)u(s^{-1}))\in E(G\bar{\rtimes}X)\subseteq \alpha(X), \ \ \ s\in G.
\end{align*}
By $G$-equivariance of $E$ it follows that
$$\alpha_t(E(Tu(s)))=E(u(t)Tu(st^{-1}))\in\alpha(X)$$
for all $s,t\in G$, which implies $E(u(f)Tu(g))\in\alpha(X)$ for all $f,g\in C_c(G)$, i.e., $T\in G\bar{\rtimes}_{\mc{F}}X$.

Conversely, if $T\in G\bar{\rtimes}_{\mc{F}}X$, then by discreteness $E(Tu(s^{-1}))\in\alpha(X)$ so that $E(Tu(s^{-1}))u(s)=T\cdot \delta_s\in G\bar{\rtimes} X$
for all $s\in G$.  Norm density of $C_c(G)\cap A(G)$ in $A(G)$ then implies $T\cdot\psi\in G\bar{\rtimes} X$ for all $\psi\in A(G)$.

Note that in proving (\ref{e:disA(G)}) we have also shown
$$G\bar{\rtimes}_{\mc{F}}X = \{T\in G\bar{\rtimes}M\mid E(Tu(s^{-1}))\in\alpha(X), \ s\in G\}.$$

\end{example}

The von Neumann algebra analogue of \cite[Proposition 3.4]{Suz}, which holds by \cite[Remark 3.6]{Suz}, shows that for a discrete group $G$ with the AP acting on a von Neumann algebra $M$, for any $G$-invariant weak* closed subspace $X$, an element $T\in G\bar{\rtimes} M$ satisfying $E(Tu(s^{-1}))\in\alpha(X)$ for all $s\in G$ is necessarily contained in $G\bar{\rtimes}X$. Example \ref{ex:2} allows us to interpret this result as a ``slice map property'' with respect to the Fubini crossed product.

\begin{defn} A $W^*$-dynamical system $(M,G,\alpha)$ has the \textit{slice map property} if $G\bar{\rtimes}X=G\bar{\rtimes}_{\mc{F}}X$ for every $G$-invariant weak* closed subspace $X\subseteq M$. 
\end{defn}

\begin{example}\label{ex:AP} Let $G$ be a locally compact group with the AP. As we shall see, by Corollary \ref{c:actionAP} below, every $W^*$-dynamical system $(M,G,\alpha)$ has the slice map property. 

After the first version of this paper appeared, this result was obtained using different techniques by Andreou \cite[Proposition 4.3]{A2}.
\end{example}

\begin{example}\label{e:3.7} Let $G$ be a locally compact group acting trivially on $\BH$ for a separable Hilbert space $H$. By Example \ref{ex:1} we have $G\bar{\rtimes}_{\mc{F}} X=VN(G)\oten_{\mc{F}} X$ for any weak* closed subspace of $X$ of $\BH$. Hence, the trivial action has the slice map property if and only if $VN(G)$ has the dual slice map property for weak* closed subspaces of $\BH$, equivalently, $VN(G)$ has the w*OAP. 
If, in addition, $G$ is inner amenable in the sense of Paterson \cite[2.35H]{Pat2}
(e.g., $G$ is discrete), then it follows from \cite[Corollary 4.8]{C2} that $G$ has the AP.
\end{example}

\subsection{$C^*$-setting} 

We have the analogous notions in the setting of $C^*$-dynamical systems.

\begin{defn}\label{d:cstarFubini} Let $(A,G,\alpha)$ be a $C^*$-dynamical system. For a $G$-invariant closed subspace $X\subseteq A$ we define the \textit{Fubini crossed product} of $X$ by $G$ as
$$G\rtimes_{\mc{F}}X:=\{T\in G\rtimes A\mid E(u(f)Tu(g))\in \alpha(X), \ f,g\in C_c(G)\}.$$
\end{defn}

\begin{prop}\label{p:UVcstar} Let $(A,G,\alpha)$ be a $C^*$-dynamical system. For any $G$-invariant closed subspace $X\subseteq A$,
$$G\rtimes_{\mc{F}}X=\{T\in G\rtimes A\mid (\rho\ten\id)(T)\in X, \ \rho\in\TC\}.$$
\end{prop}

\begin{proof} First consider $T\in G\rtimes A$ of the form $\alpha(x)u(h)$, $x\in A$, $h\in C_c(G)$. For $f,g\in C_c(G)$, the exact same integral calculation from the proof of Proposition \ref{p:UV} shows that
$$E(u(f)^*Tu(g))=\alpha(x\star(\overline{f}\cdot(h\ast g))),$$
where $h\ast g$ is the convolution of $h$ and $g$ in $\LO$. Since $\lm(h)\cdot\om_{g,f}|_{\LI}=\overline{f}\cdot(h\ast g)\in\LO$, we have
$$(\om_{g,f}\ten\id)(\alpha(x)u(h))=(\lm(h)\cdot\om_{g,f}|_{\LI}\ten\id)(\alpha(x))=x\star(\overline{f}\cdot(h\ast g)).$$
Thus, 
$$\alpha((\om_{g,f}\ten\id)(T))=E(u(f)^*Tu(g)).$$
Since the left hand side is norm continuous in $T$, and the right hand side is weak* continuous for $T$ in bounded sets (viewing $E$ as an operator-valued weight on $(G\rtimes A)''$), the above equality is valid for every $T\in G\rtimes A$, establishing the claim as in Proposition \ref{p:UV}.
\end{proof}

\begin{prop} Let $(A,G,\alpha)$ be a $C^*$-dynamical system. For any $G$-invariant closed subspace $X\subseteq A$ we have
$$G\rtimes X\subseteq G \rtimes_{\mc{F}}X,$$
where $G\rtimes X=\overline{\la\alpha(X)(C^*_{\lm}(G)\ten 1)\ra}^{\norm{\cdot}}$.
\end{prop}

\begin{proof} By the proof of Proposition \ref{p:UVcstar}, for every $f,g\in C_c(G)$, 
$$\alpha((\om_{g,f}\ten\id)(T))=E(u(f)^*Tu(g)), \ \ \ T\in G\rtimes A.$$
In particular, $E(u(f)^*(\cdot)u(g))$ is norm continuous in $T$, and for $T=\alpha(x)u(h)$, $x\in X$, $h\in C_c(G)$, 
$$E(u(f)^*Tu(g))=\alpha(x\star(\overline{f}\cdot(h\ast g)))\in\alpha(X)$$
by $\LO$-invariance of $X$. The norm continuity of $E(u(f)^*(\cdot)u(g))$ ensures the same is true for arbitrary $T\in G \rtimes X$. 

\end{proof}

\begin{defn} Let $(A,G,\alpha)$ be a $C^*$-dynamical system. We say that the action $\alpha$ has the \textit{slice map property} if $G \rtimes X=G\rtimes_{\mc{F}}X$ for every $G$-invariant closed subspace $X\subseteq A$. 
\end{defn}

\begin{example} Let $G$ be a locally compact group with the AP. As we shall see, by Corollary \ref{c:integralrepC*} below, every $C^*$-dynamical system $(A,G,\alpha)$ has the slice map property.
\end{example}

Following \cite[Definition 1.5]{Sie}, a $C^*$-dynamical system $(A,G,\alpha)$ is \textit{exact} if for every $G$-invariant (norm-closed two-sided) ideal $I\lhd A$ the sequence
$$0\rightarrow G\rtimes I\hookrightarrow G\rtimes A\twoheadrightarrow G\rtimes A/I\rightarrow0$$
is exact. A locally compact group $G$ is \textit{exact} if every $C^*$-dynamical system $(A,G,\alpha)$ is exact \cite{KW}.

\begin{prop}\label{p:exactaction} Let $(A,G,\alpha)$ be a $C^*$-dynamical system. If $(A,G,\alpha)$ has the slice map property then it is exact.
\end{prop}

\begin{proof} We follow similar lines to \cite[Proposition 1.6]{Sie}. Let $I\lhd A$ be a $G$-invariant ideal, and let $q:A\twoheadrightarrow A/I$ denote the quotient map. We show that $J:=\mathrm{Ker}(\id\rtimes q)\subseteq G\rtimes I$. 

Fix $f,g\in C_c(G)$ and let $E^A_{f,g}:=\alpha_A^{-1}(E^A(u(f)(\cdot)u(g)))$ and $E^{A/I}_{f,g}:= \alpha_{A/I}^{-1}(E^{A/I}(u(f)(\cdot)u(g)))$. Then $E^A_{f,g}:G\rtimes A\rightarrow A$, and by the proof of Proposition \ref{p:UVcstar}
$$E^A_{f,g}(T)=\alpha_A^{-1}(E^A(u(f^o)^*Tu(g)))=(\om_{g,f^o}\ten\id)(T), \ \ \ T\in G\rtimes A.$$
Similarly for $E^{A/I}_{f,g}$. Hence, for any $T\in G\rtimes A$,
$$E^{A/I}_{f,g}((\id\rtimes q)(T))=(\om_{g,f^o}\ten\id)((\id\rtimes q)(T))=q((\om_{g,f^o}\ten\id)(T))=q\circ E^A_{f,g}(T).$$
Thus, $E^A_{f,g}(J)\subseteq\mathrm{Ker}(q)=I$. Since $f$ and $g$ in $C_c(G)$ were arbitrary and the action has the slice map property,
$$J\subseteq G\rtimes_{\mc{F}} I=G\rtimes I.$$
\end{proof}

\section{Fej\'{e}r Representations in Crossed Products}

We now establish our Fej\'{e}r representation for arbitrary elements in $C^*$- and von Neumann crossed products by locally compact groups with the AP. Throughout the proof we adopt the standard leg notation for fundamental unitaries, e.g., $W_{12}=(W\ten 1)$, $W_{23}=(1\ten W)$, etc.

\begin{thm}\label{t:Landstad} Let $G$ be a locally compact group, and denote by $(f_i)\subseteq C_c(G)_{\norm{\cdot}_1=1}^+$ a symmetric bounded approximate identity for $\LO$. Then the following are equivalent.
\begin{enumerate}
\item $G$ has the AP.
\item There exists a net $(h_j)$ in $A(G) \cap C_c(G)$ such that for every $W^*$-dynamical system $(M,G,\alpha)$,
\begin{eqnarray}\label{e:intrep} T=w^*-\lim_j\lim_i\int_G \frac{h_j(x)}{\Delta(x)}E(u(f_i)Tu(f_i)u(x^{-1}))u(x) \ dx, \ \ \ T\in G\bar{\rtimes} M.
\end{eqnarray}
\end{enumerate}
\end{thm}

\begin{remark} 
If $G$ is weakly amenable, then the net $(h_i)$ can be chosen bounded in the $\Mcb$-norm, and if $G$ is amenable, bounded in the $A(G)$-norm. 
\end{remark} 

\begin{remark} When $G$ is discrete, putting $f_i=\delta_e$, the Fej\'{e}r representation (\ref{e:intrep}) simplifies to 
$$T=w^*-\lim_j\sum_{x\in G} h_j(x)E(Tu(x^{-1}))u(x).$$
\end{remark}

The first step in the proof of Theorem \ref{t:Landstad} is based on Lemma \ref{l:Landstad} below, which in turn is a version of \cite[Lemma 2.9]{L}.

\begin{lem}\label{l:Landstad} Let $(M,G,\alpha)$ be a $W^*$-dynamical system. For $T\in G\bar{\rtimes} M$, $f,g,h\in C_c(G)$, and $\xi,\eta\in \LT\ten H$, we have 
\begin{equation}\label{e:2}\int_G h(x)\la E(u(f)Tu(g)u(x^{-1}))u(x)\xi,\eta\ra=\la\h{\alpha}(u(f)T)(W\ten1)(f\ten \xi),\Delta\cdot\overline{h}\ten\eta\ra.\end{equation}
Consequently,
$$\int_G h(x) E(u(f)Tu(g)u(x^{-1}))u(x) \ dx=(\om_{f,\Delta\cdot\overline{h}}\ten\id)(\h{\alpha}(u(f)T)(W\ten1))\in\mc{B}(\LT\ten H).$$
\end{lem}

\begin{proof} First consider the case when $T=\alpha(m)u(y)$ for some $m\in M$ and $y\in G$. Then for each $x\in G$,
\begin{align*}&\la E(u(f)Tu(g)u(x^{-1}))u(x)\xi,\eta\ra=\la(\vphi\ten\id\ten\id)(\h{\alpha}(u(f)Tu(g)u(x^{-1}))(1\ten u(x)))\xi,\eta\ra\\
&=\la(\vphi\ten\id\ten\id)(\h{\alpha}(u(f)Tu(g))(\lm(x^{-1})\ten1\ten1))\xi,\eta\ra\\
&=\bigg\la(\vphi\ten\id\ten\id)\bigg(\iint f(s)g(t)(\lm(s)\ten u(s))\h{\alpha}(T)(\lm(tx^{-1})\ten u(t)) \ ds \ dt\bigg)\xi,\eta\bigg\ra\\
&=\bigg\la(\vphi\ten\id\ten\id)\bigg(\iint f(s)g(t)(1\ten u(s)\alpha(m)u(yt))(\lm(sytx^{-1})\ten1\ten1) \ ds \ dt\bigg)\xi,\eta\bigg\ra\\
&=\vphi\bigg(\iint f(s)g(t)\la u(s)\alpha(m)u(yt)\xi,\eta\ra\lm(sytx^{-1}) \ ds \ dt\bigg)\\
&=\Delta(x)\vphi\bigg(\iint f(s)g(y^{-1}s^{-1}tx)\la u(s)\alpha(m)u(s^{-1}tx)\xi,\eta\ra\lm(t) \ ds \ dt\bigg)\\
&=\Delta(x)\int f(s)g(y^{-1}s^{-1}x)\la u(s)\alpha(m)u(s^{-1}x)\xi,\eta\ra \ ds\\
&=\Delta(x)\int f(s)g(y^{-1}s^{-1}x)\la u(s)Tu(y^{-1}s^{-1}x)\xi,\eta\ra \ ds.
\end{align*}
Thus,
\begin{align*}
&\int h(x)\la E(u(f)Tu(g)u(x^{-1}))u(x)\xi,\eta\ra \ dx\\
&=\iint \Delta\cdot h(x) f(s)g(y^{-1}s^{-1}x)\la u(s)Tu(y^{-1}s^{-1}x)\xi,\eta\ra \ ds \ dx\\
&=\iiint \Delta\cdot h(x) f(s)g(y^{-1}s^{-1}x)\la (u(y^{-1}s^{-1}x)\xi)(t),(T^*u(s)^*\eta)(t)\ra \ dt \ ds \ dx\\
&=\iiint f(s)\la W_{12}(g\ten \xi)(y^{-1}s^{-1}x,t),(1\ten T^*u(s)^*)(\Delta\cdot\overline{h}\ten\eta)(x,t)\ra \ dt \ ds \ dx\\
&=\iiint f(s)\la(\lm(sy)\ten1\ten1)W_{12}(g\ten\xi)(x,t),(1\ten T^*u(s)^*)(\Delta\cdot\overline{h}\ten\eta)(x,t)\ra \ dt \ ds \ dx\\
&=\int f(s)\la(\lm(sy)\ten1\ten1)W_{12}(g\ten\xi),(1\ten T^*u(s)^*)(\Delta\cdot\overline{h}\ten\eta)\ra \ ds\\
&=\int f(s)\la(\lm(sy)\ten u(s)T)W_{12}(g\ten\xi),\Delta\cdot\overline{h}\ten\eta\ra \ ds\\
&=\int f(s)\la\h{\alpha}(u(s)T)W_{12}(g\ten\xi),\Delta\cdot\overline{h}\ten\eta\ra \ ds\\
&=\la\h{\alpha}(u(f)T)W_{12}(g\ten\xi),\Delta\cdot\overline{h}\ten\eta\ra.
\end{align*}
By norm continuity of $E(u(f)(\cdot)u(g))$, the result holds for $T$ belonging to $C^*(\alpha(M),u(G))$, and then an application of Kaplansky's density theorem together with the weak* continuity of $E(u(f)(\cdot)u(g))$ on bounded sets establishes the result for arbitrary $T\in G\bar{\rtimes}M$.
\end{proof}

The next step in the proof of Theorem \ref{t:Landstad} is to reinterpret the Hilbert space inner product in (\ref{e:2}) as a particular operator space duality which behaves well as we let $f$ and $h$ vary accordingly. To this end, we consider the following maps on $\LI\ten L^2(G,H)$ induced from the fundamental unitary $\h{W}$:
\begin{equation}\label{e:hW}\Phi_1(f\ten\xi)(s,t)=f(ts)\xi(t)=\Sigma\Gam(f)(s,t)\xi(t), \ \ \ s,t\in G,\end{equation}
and
\begin{equation}\label{e:W}\Phi_2(f\ten\xi)(s,t)=f(t^{-1}s)\xi(t)=(\id\ten\kappa)\Sigma\Gam(f)(s,t)\xi(t), \ \ \ s,t\in G,\end{equation}
where $\kappa(f)(s)=f(s^{-1})$ is the co-involution on $\LI$. Note that on $C_c(G)\ten L^2(G,H)$, the maps $\Phi_1$ and $\Phi_2$ coincide with $(\h{W}\ten1)$ and $(\h{W}^*\ten1)$, respectively.

\begin{lem}\label{l:W} Let $G$ be a locally compact group and $H$ be a Hilbert space. Then $\Phi_1$ induces a complete contraction
$$\Phi_1:\Mcb\pten L^2(G,H)_c\rightarrow\LI\pten L^2(G,H)_c,$$
where $ L^2(G,H)_c$ refers to the column operator space structure on the Hilbert space tensor product $\LT\ten H$.
\end{lem}

\begin{proof} It is well-known that $\Sigma\Gam:\Mcb\rightarrow\LI\whten\LI$ is a complete isometry (see \cite[Corollary 5.5]{Spronk} for the case of $\Gam$). Let $\ev:\LI\ten(\LT\ten H)\rightarrow\LT\ten H$ be the evaluation map $(f,\xi)\mapsto (M_f\ten 1)\xi$. Then $\ev$ extends to a complete contraction $\LI\whten L^2(G,H)_c\rightarrow L^2(G,H)_c$. Indeed, it is given by the following composition
\begin{align*}\LI\whten L^2(G,H)_c&\subseteq(\BLT\ten 1_H)\whten L^2(G,H)_c\subseteq\mc{B}(\LT\ten H)\whten L^2(G,H)_c\\
&=( L^2(G,H)_c\whten L^2(G,H)_c^*)\whten L^2(G,H)_c \\
&= L^2(G,H)_c\whten( L^2(G,H)_c^*\whten L^2(G,H)_c)\\
&\cong L^2(G,H)_c\whten( L^2(G,H)_c^*\hten L^2(G,H)_c) \ \ \ \textnormal{by \cite[Corollary 3.5]{BS}}\\
&\cong L^2(G,H)_c\whten( L^2(G,H)_c^*\pten L^2(G,H)_c) \ \ \ \textnormal{by \cite[Proposition 9.3.2]{ER}}\\
&\rightarrow  L^2(G,H)_c,
\end{align*}
where in the last line we apply the dual pairing $ L^2(G,H)_c^*\pten L^2(G,H)_c\rightarrow\C$. By \cite[Proposition 3.7]{BS} it follows that
$$(\id_{\LI}\ten\ev):\LI\whten(\LI\whten L^2(G,H)_c)\rightarrow \LI\whten L^2(G,H)_c$$
is completely contractive, and hence, by the canonical identifications used above,
$$(\id_{\LI}\ten\ev):\LI\whten(\LI\whten\LTc)\rightarrow \LI\pten L^2(G,H)_c$$
is completely contractive. Since $\Mcb$ is a dual operator space, we have 
$$\Mcb\whten L^2(G,H)_c\cong\Mcb\hten L^2(G,H)_c\cong\Mcb\pten L^2(G,H)_c$$
completely isometrically, appealing to \cite[Corollary 3.5]{BS} and \cite[Proposition 9.3.2]{ER} once again. Hence, \cite[Proposition 3.7]{BS} entails that
$$(\Sigma\Gam\ten\id_{ L^2(G,H)_c}):\Mcb\pten L^2(G,H)_c\rightarrow\LI\whten(\LI\whten L^2(G,H)_c)$$
is completely contractive. From (\ref{e:hW}) it follows that 
$$\Phi_1=(\id_{\LI}\ten\ev)\circ(\Sigma\Gam\ten\id_{ L^2(G,H)_c})$$
on $\Mcb\ten L^2(G,H)_c$. It therefore extends to a complete contraction
$$\Mcb\pten L^2(G,H)_c\rightarrow\LI\pten L^2(G,H)_c.$$
\end{proof}

\begin{lem}\label{l:W2} Let $G$ be a locally compact group and $H$ be a Hilbert space. Then $\Phi_2$ induces a complete contraction
$$\Phi_2:\LI\pten L^2(G,H)_c\rightarrow\LI\oten L^2(G,H)_c.$$
\end{lem}

\begin{proof} Let $\pi_1:\LI\rightarrow\mc{CB}(\LI)$ and $\pi_2:\LI\rightarrow\mc{CB}( L^2(G,H)_c)$ be the representations given by left multiplication, where the latter maps $f\mapsto M_f\ten 1$. Since $\pi_1$ and $\pi_2$ are both weak*-weak* continuous, the map $\pi_1\ten\pi_2$ extends to a complete contraction
$$\pi_1\ten\pi_2:\LI\oten\LI\rightarrow\mc{CB}(\LI)\oten\mc{CB}( L^2(G,H)_c).$$
Since $\LI\pten\LO$ has the OAP, its dual $\mc{CB}(\LI)=(\LI\pten\LO)^*$ has the dual slice map property so that
\begin{align*}\mc{CB}(\LI)\oten\mc{CB}( L^2(G,H)_c)&=(\LI\pten\LO)^*\oten( L^2(G,H)_c\pten L^2(G,H)_c^*)^*\\
&=((\LI\pten\LO)\pten( L^2(G,H)_c\pten L^2(G,H)_c^*))^*\\
&=((\LI\pten L^2(G,H)_c)\pten(\LO\pten L^2(G,H)_c^*))^*\\
&=\mc{CB}(\LI\pten L^2(G,H)_c,\LI\oten  L^2(G,H)_c).
\end{align*}
By universality of the operator space projective tensor product it follows that $\pi_1\ten\pi_2$ induces a complete contraction 
$$m:(\LI\oten\LI)\pten(\LI\pten L^2(G,H)_c)\rightarrow\LI\oten  L^2(G,H)_c.$$

By equation (\ref{e:W}), we see that 
$$\Phi_2=m\circ((\id\ten\kappa)\Sigma\Gam\ten\id_{\LI\oten L^2(G,H)_c})\circ i$$ 
on $\LI\ten L^2(G,H)_c$, where 
$$i:\LI\pten L^2(G,H)_c\ni f\ten\eta\mapsto f\ten (1\ten\eta)\in\LI\pten(\LI\oten L^2(G,H)_c)$$
and
$$((\id\ten\kappa)\Sigma\Gam\ten\id):\LI\pten(\LI\oten L^2(G,H)_c)\rightarrow(\LI\oten\LI)\pten(\LI\oten L^2(G,H)_c)$$
are complete contractions.
\end{proof}

\begin{cor}\label{c:bounded} Let $G$ be a locally compact group and $H$ be a Hilbert space. For $T\in \mc{B}(L^2(G,H))$, 
$$\norm{\Phi_2\circ(1\ten T)\circ\Phi_1:\Mcb\pten L^2(G,H)_c\rightarrow\LI\oten L^2(G,H)_c}_{cb}\leq \norm{T}.$$
\end{cor}

\begin{proof}[Proof of Theorem \ref{t:Landstad}] Suppose $G$ has the AP. Note that for any $h\in A(G)\cap C_c(G)$, 
$$\om_{f_i,\overline{h}}|_{VN(G)}=h\ast\check{f_i}=h\ast f_i\in A(G),$$
and $h\ast f_i(s)=\la\lm(s)\lm(f_i)\eta,\zeta\ra$, $s\in G$, where $\la\lm(\cdot)\eta,\zeta\ra$ is a representation of $h$. Hence
$$\norm{h\ast f_i-h}_{A(G)}\rightarrow 0.$$
By the AP pick a net $(h_j)$ of real-valued functions in $A(G)\cap C_c(G)$ converging to 1 in the weak* topology of $\Mcb$. 

Fix $T\in G\bar{\rtimes}M$ and a positive $\rho\in\mc{T}(\LT\ten H)$. Write $\rho=\sum_{n=1}^\infty \xi_n\xi_n^*$ for a sequence $(\xi_n)$ of vectors in $\LT\ten H$ satisfying $\sum_{n=1}^\infty\norm{\xi_n}^2=\norm{\rho}$.

For every $n\in\N$ we have
$$\norm{W_{12}(f_i\ten\xi_n)-f_i\ten\xi_n}_{\LO\ten^\gamma\overline{\LT\ten H}}=\int_G f_i(s)\norm{(\lm(s)\ten 1)\xi_n-\xi_n} \ ds,$$
where $\ten^\gamma$ is the Banach space projective tensor product, so by Jensen's inequality
$$\norm{W_{12}(f_i\ten\xi_n)-f_i\ten\xi_n}_{\LO\ten^\gamma\overline{\LT\ten H}}^2\leq\int_G f_i(s)\norm{(\lm(s)\ten 1)\xi_n-\xi_n}^2 \ ds.$$
Given $\ep>0$, pick $N\in\N$ such that $\sum_{n=N+1}^\infty\norm{\xi_n}^2<\ep/8$. The function 
$$g_\ep(s)=\sum_{n=1}^N\norm{(\lm(s)\ten 1)\xi_n-\xi_n}^2$$
is right uniformly continuous and satisfies $g_\ep(e)=0$. Hence, 
$$\int_G f_i(s)g_\ep(s)=g_\ep\ast\check{f_i}(e)=g_\ep\ast f_i(e)\rightarrow g_\ep(e)=0.$$
Pick $i_\ep$ such that
$$\int_G f_i(s)g_\ep(s)<\frac{\ep}{2}$$
for all $i\geq i_\ep$. Then for all $i\geq i_\ep$ we have
\begin{align*}&\sum_{n=1}^\infty\norm{W_{12}(f_i\ten\xi_n)-f_i\ten\xi_n}_{\LO\ten^\gamma\overline{\LT\ten H}}^2\\
&\leq\sum_{n=1}^\infty\int_G f_i(s)\norm{(\lm(s)\ten 1)\xi_n-\xi_n}^2 \ ds\\
&=\sum_{n=1}^N\int_G f_i(s)\norm{(\lm(s)\ten 1)\xi_n-\xi_n}^2 \ ds+\sum_{n=N+1}^\infty\int_G f_i(s)\norm{(\lm(s)\ten 1)\xi_n-\xi_n}^2 \ ds\\
&\leq\int_G f_i(s)\bigg(\sum_{n=1}^N\norm{(\lm(s)\ten 1)\xi_n-\xi_n}^2\bigg) ds+\sum_{n=N+1}^\infty 4\norm{\xi_n}^2\\
&<\int_G f_i(s)g_\ep(s) \ ds +\frac{\ep}{2}\\
&<\ep.
\end{align*}
Hence,
$$\sum_{n=1}^\infty\norm{W_{12}(f_i\ten\xi_n)-f_i\ten\xi_n}_{\LO\ten^\gamma\overline{\LT\ten H}}^2\rightarrow0.$$
Recall that on $C_c(G)\ten L^2(G,H)$, the maps $\Phi_1$ and $\Phi_2$ coincide with $(\h{W}\ten1)$ and $(\h{W}^*\ten1)$, respectively. Corollary \ref{c:bounded} then entails
\begin{align*}&|\la\h{\alpha}(u(f_i)T)(W_{12}(f_i\ten\xi_n)-(f_i\ten\xi_n)),h_j\ten\xi_n\ra|\\
&=|\la W_{12}(f_i\ten\xi_n)-(f_i\ten\xi_n),\h{W}_{12}^*(1\ten T^*u(f_i)^*))\h{W}_{12}(h_j\ten\xi_n)\ra|\\
&\leq\norm{W_{12}(f_i\ten\xi_n)-(f_i\ten\xi_n)}_{\LO\pten \overline{L^2(G,H)_r}}\norm{\Phi_2(1\ten T^*u(f_i)^*))\Phi_1(h_j\ten\xi_n)}_{\LI\oten L^2(G,H)_c}\\
&\leq\norm{W_{12}(f_i\ten\xi_n)-(f_i\ten\xi_n)}_{\LO\ten^\gamma\overline{L^2(G,H)_r}}\norm{u(f_i)}\norm{T}\norm{h_j\ten\xi_n}_{\Mcb\pten L^2(G,H)_c}\\
&\leq\norm{W_{12}(f_i\ten\xi_n)-(f_i\ten\xi_n)}_{\LO\ten^\gamma\overline{L^2(G,H)_r}}\norm{T}\norm{h_j}_{\Mcb}\norm{\xi_n}_{\LT\ten H},
\end{align*}
where $L^2(G,H)_r$ refers to the row operator space structure of $L^2(G,H)$. Hence, for every $j$, we have
\begin{align*}
&\sum_{n=1}^\infty|\la\h{\alpha}(u(f_i)T)(W_{12}(f_i\ten\xi_n)-(f_i\ten\xi_n)),h_j\ten\xi_n\ra|\\
&\leq\norm{T}\norm{h_j}_{\Mcb}\sum_{n=1}^\infty\norm{W_{12}(f_i\ten\xi_n)-(f_i\ten\xi_n)}_{\LO\ten^\gamma\overline{L^2(G,H)_r}}\norm{\xi_n}\\
&\leq\norm{T}\norm{h_j}_{\Mcb}\bigg(\sum_{n=1}^\infty\norm{W_{12}(f_i\ten\xi_n)-(f_i\ten\xi_n)}_{\LO\ten^\gamma\overline{L^2(G,H)_r}}^2\bigg)^{1/2}\bigg(\sum_{n=1}^\infty\norm{\xi_n}^2\bigg)^{1/2}\\
&\rightarrow0.
\end{align*}

Putting things together, for each $j$ 
\begin{align*}&\lim_i\int_G\frac{h_j(x)}{\Delta(x)}\la E(u(f_i)Tu(f_i)u(x^{-1}))u(x),\rho\ra \ dx=\lim_i\la(\om_{f_i,h_j}\ten\id)(\h{\alpha}(u(f_i)T)W_{12}),\rho\ra\\
&=\lim_i\sum_{n=1}^\infty\la\h{\alpha}(u(f_i)T)W_{12}(f_i\ten\xi_n),h_j\ten\xi_n\ra=\lim_i\sum_{n=1}^\infty\la\h{\alpha}(u(f_i)T)(f_i\ten\xi_n),h_j\ten\xi_n\ra\\
&=\lim_i\sum_{n=1}^\infty\la(\om_{f_i,h_j}\ten\id)\h{\alpha}(u(f_i)T)\xi_n,\xi_n\ra=\lim_i\la(h_j\ast f_i\ten\id)\h{\alpha}(u(f_i)T),\rho\ra=\la(h_j\ten\id)\h{\alpha}(T),\rho\ra,
\end{align*}
where in the last equality we used the norm continuity of the map
$$A(G)\ni\psi\mapsto(T\mapsto T\cdot\psi)\in\mc{CB}(G\bar{\rtimes}M).$$
Now, $((h_j\ten\id)\h{\alpha}(T)=T\cdot h_j=(\h{\Theta}(h_j)\ten\id)(T)$. Letting $m_{h_j}=\h{\Theta}(h_j)|_{VN(G)}\in\mc{CB}^\sigma(VN(G))$ denote the canonical multiplication map on $VN(G)$, it follows from \cite[Proposition 1.7, Theorem 1.9(b)]{HK} that $(m_{h_j}\ten\id_N)\rightarrow\id_{VN(G)\oten N}$ point weak* for any von Neumann algebra $N$. By the proof of \cite[Proposition 4.3]{JNR} we have
$$(\h{\Theta}(h_j)\ten\id)(S)=(\om_0\ten\id_{\BLT\oten M})(\h{W}_{12}((m_{h_j}\ten\id_{\BLT\oten M})(\h{W}_{12}^*S_{23}\h{W}_{12}))\h{W}_{12}^*)$$
for all $S\in\BLT\oten M$, where $\om_0\in\TC$ is an arbitrary state. Hence,  
$$\lim_j\la(h_j\ten\id)\h{\alpha}(T),\rho\ra=\lim_j\la (\h{\Theta}(h_j)\ten\id)(T),\rho\ra=\la T,\rho\ra$$
By polarization, the above analysis is valid for any trace class operator. Hence,
$$T=w^*-\lim_j\lim_i\int\frac{h_j(x)}{\Delta(x)}E(u(f_i)Tu(f_i)u(x^{-1}))u(x) \ dx.$$

Conversely, suppose (2) holds for every $W^*$-dynamical system $(M,G,\alpha)$. If $G$ acts trivially on a von Neumann algebra $M$, then for every $T\in G\bar{\rtimes} M=VN(G)\oten M$ we have 
$$ T=w^*-\lim_j\lim_i\int_G \frac{h_j(x)}{\Delta(x)}E(u(f_i)Tu(f_i)u(x^{-1}))u(x) \ dx.$$
As in the proof of Theorem \ref{t:Landstad}, for each $j$ it follows that
$$(\h{\Theta}(h_j)\ten\id)(T)=w^*-\lim_i\int_G\frac{h_j(x)}{\Delta(x)} E(u(f_i)Tu(f_i)u(x^{-1}))u(x) \ dx.$$
Thus, $T=w^*-\lim_j (\h{\Theta}(h_j)\ten\id)(T)$, implying that $(\h{\Theta}(h_j))$ converges to $\id_{VN(G)}$ in the stable point weak* topology. By \cite[Theorem 1.9 (b)]{HK}, $G$ has the AP.
\end{proof}

As an immediate application, we obtain:

\begin{cor}\label{c:actionAP} Let $G$ be a locally compact group with the AP. Then every action of $G$ on a von Neumann algebra has the slice map property.
\end{cor}

For inner amenable groups $G$ (e.g., $G$ discrete) the converse of Corollary \ref{c:actionAP} holds by Example \ref{e:3.7} and \cite[Corollary 4.8]{C2}. For discrete groups a similar result was also shown in \cite{Suz}. After the first version of this paper appeared, the converse of Corollary \ref{c:actionAP} was established by Andreou \cite[Theorem 5.12]{A}. Hence, a locally compact group $G$ as the AP if and only if every $W^*$-dynamical system has the slice map property.


Analogous results hold in the setting of $C^*$-dynamical systems, where we obtain norm convergence of the Fej\'{e}r representation. The proof relies on the following Lemma, which is inspired by \cite[Lemma 2.4]{Zach1}. In preparation, for every $T\in G\rtimes A$ and $\vphi\in(G\rtimes A)^*$, let $\om_{T,\vphi}:\Mcb\rightarrow\C$ be the functional
$$\om_{T,\vphi}(v)=\la(\h{\Theta}(v)\ten\id)(T),\vphi\ra, \ \ \ v\in\Mcb.$$

\begin{lem}\label{l:qcb} For every $T\in G\rtimes A$, and $\vphi\in(G\rtimes A)^*$, the functional $\om_{T,\vphi}\in\Qcb$.
\end{lem}

\begin{proof} First suppose that $T=\tilde{\pi}_u\times u(f)$ for some $f\in C_c(G,A)$, and $\vphi\in(G\rtimes A)^*$ is positive. By \cite[Lemme 3.2]{E} there exists $v_f\in A(G)\cap C_c(G)$ in the span of compactly supported positive definite functions satisfying $v_f\equiv 1$ on $\mathrm{supp}(f)$. It follows that $T=(\h{\Theta}(v_f)\ten\id)(T)$. Also, by \cite[Corollary 7.6.9, Lemma 7.7.6]{Ped} we have $(\h{\Theta}(v_f)\ten\id)_*(\vphi)\in(G\rtimes A)''_*$. 

If $v_i\rightarrow v$ in $\sigma(\Mcb,\Qcb)$ then, as above, $(\h{\Theta}(v_i)\ten\id)(S)\rightarrow(\h{\Theta}(v)\ten\id)(S)$ weak* for all $S\in\BLT\oten\mc{B}(H_u)$, so that
\begin{align*}\om_{T,\vphi}(v_i)&=\la(\h{\Theta}(v_i)\ten\id)(T),\vphi\ra=\la(\h{\Theta}(v_i)\ten\id)(T),(\h{\Theta}(v_f)\ten\id)_*(\vphi)\ra\\
&\rightarrow\la(\h{\Theta}(v)\ten\id)(T),(\h{\Theta}(v_f)\ten\id)_*(\vphi)\ra=\la(\h{\Theta}(v)\ten\id)(T),\vphi\ra.
\end{align*}
Hence, $\om_{T,\vphi}\in\Qcb$. By norm density of $\tilde{\pi}_u\times u(C_c(G,A))$ in $G\rtimes A$ and polarization, we see that $\om_{T,\vphi}\in\Qcb$ for every $T\in G\rtimes A$, and $\vphi\in(G\rtimes A)^*$.

\end{proof}

\begin{thm}\label{t:cstarLandstad} Let $G$ be a locally compact group, and denote by $(f_i)\subseteq C_c(G)_{\norm{\cdot}_1=1}^+$ a symmetric bounded approximate identity for $\LO$. Then the following are equivalent.
\begin{enumerate}
\item $G$ has the AP.
\item There exists a net $(h_j)$ in $A(G) \cap C_c(G)$ such that for every $C^*$-dynamical system $(A,G,\alpha)$,
\begin{eqnarray}\label{e:intrepCstar} T=\lim_j\lim_i\int_G \frac{h_j(x)}{\Delta(x)}E(u(f_i)Tu(f_i)u(x^{-1}))u(x) \ dx, \ \ \ T\in G\rtimes A.
\end{eqnarray}
\end{enumerate}
\end{thm}

\begin{proof} The proof goes through more or less unchanged from that of Theorem \ref{t:Landstad}. The final convergence simply needs to be upgraded to the norm topology. If $G$ has the AP, let $(h_j)$ be a net of real-valued functions in $A(G)\cap C_c(G)$ converging to 1 in the weak* topology of $\Mcb$. Then as above, for each $j$ we have
$$\lim_i\int_G\frac{h_j(x)}{\Delta(x)}\la E(u(f_i)Tu(f_i)u(x^{-1}))u(x),\rho\ra \ dx=\la(h_j\ten\id)\h{\alpha}(T),\rho\ra=\la(\h{\Theta}(h_j)\ten\id)(T),\rho\ra,$$
for all $T\in G\rtimes A$ and $\rho\in\mc{T}(\LT\ten H_u)$, where the convergence is uniform in $\rho$. Hence, 
$$\lim_i\int_G\frac{h_j(x)}{\Delta(x)} E(u(f_i)Tu(f_i)u(x^{-1}))u(x) \ dx=(\h{\Theta}(h_j)\ten\id)(T)$$
in the norm topology of $\mc{B}(\LT\ten H_u)$. By Lemma \ref{l:qcb}
$$\la(\h{\Theta}(h_j)\ten\id)(T),\vphi\ra=\om_{T,\vphi}(h_j)\rightarrow\om_{T,\vphi}(1)=\la T,\vphi\ra$$ 
for any $\vphi\in(G\rtimes A)^*$, that is, $(\h{\Theta}(h_j)\ten\id)(T)\rightarrow T$ weakly in $G\rtimes A$. Taking convex combinations of the $h_j$, we may assume that the later convergence is relative to the norm topology. Hence,
$$\lim_j\lim_i\int_G\frac{h_j(x)}{\Delta(x)}E(u(f_i)Tu(f_i)u(x^{-1}))u(x) \ dx = T$$
in the norm topology of $G\rtimes A$.

Conversely, suppose (2) holds for every $C^*$-dynamical system $(A,G,\alpha)$, and let $G$ act trivially on $\mc{K}(H)$ for a separable Hilbert space $H$. Then for every $T\in G\rtimes A=C_\lm^*(G)\iten \mc{K}(H)$ we have 
$$ T=\lim_j\lim_i\int_G \frac{h_j(x)}{\Delta(x)}E(u(f_i)Tu(f_i)u(x^{-1}))u(x) \ dx.$$
As in the proof of Theorem \ref{t:cstarLandstad}, for each $j$ it follows that
$$(\h{\Theta}(h_j)\ten\id)(T)=\lim_i\int_G\frac{h_j(x)}{\Delta(x)} E(u(f_i)Tu(f_i)u(x^{-1}))u(x) \ dx.$$
Thus, $T=\lim_j (\h{\Theta}(h_j)\ten\id)(T)$, implying that $(\h{\Theta}(h_j))$ converges to $\id_{C_\lm^*(G)}$ in the stable point norm topology. By \cite[Theorem 1.9 (c)]{HK}, $G$ has the AP.
\end{proof}

\begin{cor}\label{c:integralrepC*} Let $G$ be a locally compact group with the AP. Then any action of $G$ on any $C^*$-algebra has the slice map property. In particular, $C^*_\lm(G)$ has the strong operator space approximation property (SOAP).
\end{cor}  

At a similar time as the first version of this paper, it was shown by Suzuki that for any $C^*$-dynamical system $(A,G,\alpha)$ where $G$ has the AP, $G\rtimes A$ has the SOAP if and only if $A$ has the SOAP \cite[Theorem 3.6]{Suz2}. In particular, this yields a different proof that $C^*_\lm(G)$ has the SOAP for any locally compact group $G$ with the AP. Contrary to the $W^*$-setting (see Corollary \ref{c:actionAP} and the subsequent discussion), the converse of Corollary \ref{c:integralrepC*} remains open in general.



\section{Applications}

\subsection{The approximation property and exactness} 

It is well-known that a discrete group $G$ with the AP is exact \cite{HK}. The argument proceeds by establishing the SOAP of $C^*_\lm(G)$, then appealing to the fact that $C^*$-algebras with the SOAP are exact, 
and the fact that a discrete group $G$ with $C^*_\lm(G)$ exact is necessarily exact \cite[Theorem 5.2]{KW}. To the authors' knowledge, a similar argument through the operator space structure of $C^*_\lm(G)$ cannot be applied in the general locally compact setting. However, combining Corollary \ref{c:integralrepC*} with Proposition \ref{p:exactaction} we see that the implication persists. This answers Problem 9.4 (1) in \cite{Li}. 

\begin{thm}\label{t:APexact} A locally compact group with the AP is exact.\end{thm}

\begin{remark} The result of Theorem \ref{t:APexact} also appears in the independent very recent work of Suzuki \cite{Suz2}, of which we became aware after obtaining this result. 
\end{remark}

\begin{remark} It was shown in \cite[Corollary E]{BCL} that a weakly amenable second countable locally compact group is exact. We therefore obtain a different proof of this fact, valid for any locally compact group.\end{remark} 

\subsection{The Fej\'{e}r property for discrete dynamical systems} Let $(A,G,\alpha)$ be a $C^*$-dynamical system with $G$ discrete. Following \cite{BC1}, a function $F:G\times A\rightarrow A$, linear in the second variable, is a (completely) bounded multiplier of $(A,G,\alpha)$ if
$$M_{F}:G\rtimes A\ni \sum_{x\in G}a_x u(x)\mapsto \sum_{x\in G}F(x,a_x)u(x)\in G\rtimes A$$
extends to a (completely) bounded map. A function $F:G\times A\rightarrow A$ has \textit{finite $G$-support} if $F(x,\cdot)=0$ for all but finitely many $x\in G$. The $C^*$-dynamical system $(A,G,\alpha)$ has the \textit{Fej\'{e}r property} if there exists a net $(F_i)$ of bounded multipliers for which each $F_i$ has finite $G$-support and $M_{F_i}(T)\rightarrow T$ in norm for all $T\in G\rtimes A$. In \cite[Theorem 5.6]{BC1} the authors show that if $G$ is weakly amenable, then any $C^*$-dynamical system has the Fej\'{e}r property, and they ask whether the corresponding result is true for groups with the AP. 

\begin{thm}\label{t:BC} Let $G$ be a discrete group with the AP. Then every $C^*$-dynamical system $(A,G,\alpha)$ has the Fej\'{e}r property.
\end{thm}

\begin{proof} By Theorem \ref{t:cstarLandstad} there exists a net $(h_i)\subseteq A(G)\cap C_c(G)$ such that  
$$T=\lim_j\sum_{x\in G} h_j(x)E(Tu(x^{-1}))u(x), \ \ \ T\in G\rtimes A.$$
Letting $F_j(x,a)=h_j(x)a$, $x\in G$, $a\in A$, we obtain a net $(F_j)$ of finitely $G$-supported functions $G\times A\rightarrow A$ whose corresponding multipliers $M_{F_j}$ equal $(\Theta(h_j)\ten\id)$ (this is easily verified on finitely supported elements of the form $\sum_{x\in G}a_x u(x)$, the extension to $G\rtimes A$ following from boundedness). The proof of Theorem \ref{t:cstarLandstad} then shows
$$M_{F_i}(T)=(\Theta(h_j)\ten\id)(T)=\sum_{x\in G}h_j(x)E(Tu(x^{-1}))u(x), \ \ \ T\in G\rtimes A.$$
Hence, $M_{F_i}(T)\rightarrow T$ for all $T\in G\rtimes A$ and $(A,G,\alpha)$ has the Fej\'{e}r property.
\end{proof}

Alternatively, Theorem \ref{t:BC} can be deduced from the proof of \cite[Proposition 3.4]{Suz}, as pointed out to us by Suzuki, who learned about this question and our solution through an earlier version of the present paper.

It is natural to consider a generalization of the Fej\'{e}r property for locally compact groups using the recent theory of Herz--Schur multipliers of dynamical systems \cite{MTT}, and connect this to our slice map property for actions. This and related questions will appear in subsequent work.

\subsection{Structure of $VN(G)$-bimodules}

Given a locally compact group $G$ and a closed left ideal $J\lhd\LO$, following \cite{AKT2}, we let
$$\mathrm{Ran}(J)=\overline{\mathrm{span}\{\Theta^r(f)_*(\rho)\mid f\in J, \ \rho\in\TC\}}^{\norm{\cdot}_{\TC}}$$
and
$$\mathrm{Bim}(J^{\perp})=\overline{\mathrm{span}\{xM_fy\mid f\in J^{\perp}, \ x,y\in VN(G)\}}^{w^*},$$
where $\Theta^r(f)_*$ is the pre-adjoint of the normal completely bounded map
$$\Theta^r(f):\BLT\ni T\mapsto \int_G f(s)\rho(s)T\rho(s^{-1}) \ ds\in\BLT.$$
In \cite{AKT2} it was shown that for any $G$ and $J\lhd\LO$, $\mathrm{Bim}(J^{\perp})\subseteq \mathrm{Ran}(J)^{\perp}$. The reverse inclusion was established in a few special cases, including abelian, compact, and weakly amenable discrete groups. The authors asked whether it holds in general. As an application of Theorem \ref{t:Landstad}, we now show that the reverse inclusion holds for all locally compact groups with the AP.

\begin{thm}\label{t:last} Let $G$ be a locally compact group with the AP. Then for any closed left ideal $J\lhd\LO$ we have $\mathrm{Bim}(J^{\perp})=\mathrm{Ran}(J)^{\perp}$.
\end{thm}

\begin{proof} First, we make the identification $\BLT\cong G\bar{\rtimes}\LI$ via the extended (right) co-product
$$\Gam^r:\BLT\ni T\mapsto V(T\ten 1)V^*\in\BLT\oten\LI.$$
Under this identification it follows that the dual co-action $\h{\alpha}:\BLT\rightarrow VN(G)\oten\BLT$ is precisely the extended (left) co-product
$$\h{\Gam}^l:\BLT\ni T\mapsto \h{W}^*(1\ten T)\h{W}\in VN(G)\oten\BLT.$$
Next, for any $r,s,t\in G$, and $\xi\in L^2(G\times G)$ we have
\begin{align*}(\lm(r)\ten\rho(r))\h{W}\xi(s,t)&=\h{W}\xi(r^{-1}s,tr)\Delta(r)^{1/2}\\
&=\xi(ts,tr)\Delta(r)^{1/2}\\
&=(1\ten\rho(r))\xi(ts,t)\\
&=\h{W}(1\ten\rho(r))\xi(s,t).
\end{align*}
Thus, $(\lm(r)\ten\rho(r))\h{W}=\h{W}(1\ten\rho(r))$. For any $r\in G$, positive definite $\psi\in A(G)$, and $T\in\BLT$ we therefore have
\begin{align*}\mathrm{Ad}(\rho(r))((\psi\ten\id)\h{\alpha}(T))&=(\psi\ten\id)((1\ten\rho(r))\h{W}^*(1\ten T)\h{W}(1\ten\rho(r^{-1})))\\
&=(\psi\ten\id)(\h{W}^*(1\ten \rho(r)T\rho(r^{-1}))\h{W})\\
&=(\psi\ten\id)\h{\alpha}(\rho(r)T\rho(r^{-1})).\end{align*}
In the notation of \cite[Theorem 3.1 (a)]{Haa}, the above reads $E_\psi\circ\Ad(\rho(r))=\Ad(\rho(r))\circ E_\psi$. Since the operator-valued weight $E=(\vphi\ten\id)\circ\h{\alpha}$ coincides with Haagerup's construction in \cite{Haa}, it satisfies 
$$E=\sup_{\psi<<\delta_e}E_\psi,$$
where $\psi$ ranges through the positive definite elements of $A(G)$ and $<<$ is the standard positive definite ordering (see \cite[Theorem 3.1 (b), Proposition 2.4]{Haa}). It follows that for every $f,g\in C_c(G)$ and $T\in\BLT$,
$$\rho(r)E(\lm(f)T\lm(g))\rho(r^{-1})=E(\lm(f)\rho(r)T\rho(r^{-1})\lm(g)), \ \ \ r\in G,$$
and hence, by normality of $E(\lm(f)(\cdot)\lm(g))$, that
$$\Theta^r(h)(E(\lm(f)T\lm(g)))=\int_G h(r)\rho(r)E(\lm(f)T\lm(g))\rho(r^{-1})=E(\lm(f)\Theta^r(h)(T)\lm(g))$$
for all $h\in\LO$.

Now, suppose $T\in\mathrm{Ran}(J)^{\perp}$. As $\mathrm{Ran}(J)^{\perp}$ is a $VN(G)$-bimodule, we have $\lm(f)T\lm(f)\lm(x^{-1})\in\mathrm{Ran}(J)^{\perp}$ for all $f\in C_c(G)$ and $x\in G$. 
Since $\mathrm{Ran}(J)^{\perp}=\mathrm{Ker}(\Theta^r(J))$ \cite[Lemma 3.1]{AKT2}, and
$$\Theta^r(h)(E(\lm(f)T\lm(f)\lm(x^{-1})))=E(\lm(f)\Theta^r(h)(T)\lm(f)\lm(x^{-1}))=0$$
for all $h\in J$, we have $E(\lm(f)T\lm(f)\lm(x^{-1}))\in\mathrm{Ran}(J)^{\perp}\cap\LI=J^{\perp}$ by \cite[Lemma 3.3]{AKT2}. The Fej\'{e}r representation for $T$ from Theorem \ref{t:Landstad} then implies that $T\in\mathrm{Bim}(J^{\perp})$.

\end{proof}

After the first version of this paper appeared, Theorem \ref{t:last} was obtained using different techniques by Andreou \cite[Remark 5.4]{A2}.

\section*{Acknowledgements}

The authors are grateful for the reviewers' comments, which improved the overall presentation of the paper. The authors would also like to thank Mahmood Alaghmandan, Mehrdad Kalantar, Hung-Chang Liao, and Ivan Todorov for helpful discussions at various points during the project. The first author was partially supported by the NSERC Discovery Grant RGPIN-2017-06275, and the second author was partially supported by the NSERC Discovery Grant RGPIN-2014-06356. 



\end{spacing}

\end{document}